\documentclass[leqno,draft]{amsart}
\date{July 1, 2011}
\newcommand{\vect}[1]{\boldsymbol{#1}}
\renewcommand{\L}{\mathcal{L}}

\newcommand{\D}{\mathbb{D}}
\newcommand{\DD}{\mathcal{D}}
\newcommand{\C}{\mathbb{C}}
\newcommand{\R}{\mathbb{R}}

\newcommand{\A}{\mathcal{A}}
\newcommand{\SL}{\operatorname{SL}}
\newcommand{\Herm}{\operatorname{Herm}}

\newcommand{\SU}{\operatorname{SU}}
\newcommand{\PSL}{\operatorname{PSL}}

\newcommand{\id}{\operatorname{id}}
\newcommand{\diam}{\operatorname{diam}}

\newcommand{\Length}{\operatorname{Length}}

\renewcommand{\Re}{\operatorname{Re}}
\newcommand{\trace}{\operatorname{trace}}
\newcommand{\trans}[1]{\vphantom{#1}^t#1}

\newcommand{\dist}{\operatorname{dist}}
\newcommand{\imag}{\mathrm{i}}
\renewcommand{\theenumi}{(\arabic{enumi})}
\renewcommand{\labelenumi}{(\arabic{enumi})}
\numberwithin{equation}{section}
\newtheorem{theorem}{Theorem}[section]

\newtheorem{lemma}[theorem]{Lemma}

\newtheorem{proposition}[theorem]{Proposition}
\newtheorem*{maintheorem*}{Main Theorem}
\newtheorem*{mainlemma*}{Main Lemma}
\newtheorem*{keylemma*}{Key Lemma}

\theoremstyle{definition}
\newtheorem{remark}[theorem]{Remark}
\newtheorem*{remark*}{Remark}
\newtheorem{definition}[theorem]{Definition}

\title[Calabi-Yau Problem for Legendrian Curves]{%
   Calabi-Yau Problem for
   Legendrian curves in $\C^3$\\
   and applications
}%
\author[F. Martin]{Francisco Mart\'\i{}n}
\address[Martin]{
  Departamento de Geometr\'\i{}a y Topolog\'\i{}a,
  Universidad de Granada,
  18071 Granada, Spain.
}
\email{fmartin@ugr.es}
\author[M. Umehara]{Masaaki Umehara}
\address[Umehara]{%
   Department of Mathematical and Computing Sciences,
   Tokyo Institute of Technology
   2-12-1-W8-34, O-okayama, Meguro-ku,
   Tokyo 152-8552, Japan.
}
\email{umehara@is.titech.ac.jp}
\author[K. Yamada]{Kotaro Yamada}
\address[Yamada]{%
   Department of Mathematics,
   Tokyo Institute of Technology
   2-12-1-H-7, O-okayama, Meguro-ku,
   Tokyo 152-8551, Japan.
}
\email{kotaro@math.titech.ac.jp}
\thanks{
  The first author is partially
  supported by MEC-FEDER Grant no. MTM2007 - 61775 and a Regional
  J. Andaluc\'\i a Grant no. P09-FQM-5088.
  The second and the third authors  are 
  partially supported by Grant-in-Aid for 
  Scientific Research (A) No.~22244006, 
  and Scientific Research (B) No.~21340016,
  respectively, from the Japan Society for the Promotion of Science.
}

\subjclass[2010]{%
  Primary 53D10; 
  Secondary 53A10, 
            53A15, 
	    53A35. 
}

\begin{document}
\maketitle
\begin{abstract}
We  construct a complete, bounded Legendrian immersion in $\C^3$. 
As direct applications of it, we show the first examples of
a weakly complete bounded flat front in hyperbolic $3$-space,
a weakly complete bounded flat front in de Sitter $3$-space,
and a weakly complete bounded improper affine front in $\R^3$.
\end{abstract}

\section{Introduction}
In a series of previous papers, the authors have constructed
the first examples of complete bounded null holomorphic immersion
\[
   \nu:\D_1\longrightarrow \C^3
\]
of the unit disc $\D_1\subset\C$, 
where {\it null\/} means that $\nu_z\cdot \nu_z$ vanishes
identically, here $\nu_z:=d\nu/dz$ is the derivative of 
$\nu$ with respect to the complex coordinate $z$ of $\D_1$
and the dot denotes the canonical complex bilinear form.
The existence of such an immersion has important consequences.
Actually, 
as a short and direct application of the main result in \cite{MUY}, by using
different kinds of transformations,
the following objects were constructed:
\begin{enumerate}
\item complete bounded minimal surfaces in  the Euclidean 3-space $\R^3$
      (\cite[Theorem A]{MUY}),
\item complete bounded holomorphic curves in $\C^2$ 
      (\cite[Corollary~B]{MUY}).
\item weakly complete bounded maximal surfaces in the Lorentz-Minkowski
      3-space $\R_1^3$
      (\cite[Corollary D]{MUY}),
\item complete bounded constant mean curvature
      one surfaces in the hyperbolic 3-space $H^3$
      (\cite[Theorem C]{MUY}).
\end{enumerate}
Moreover, we constructed 
higher genus examples of the first three objects
in \cite{MUY2}. 
Recently, Alarc\'on and L\'opez
\cite{AL} have constructed
a complete bounded null proper holomorphic immersion
of a given Riemann surface of an arbitrary topology
into a convex domain in $\C^3$ (see also \cite{AL2}). 
Their method is different from ours. 

It is known that null curves in $\C^3$ are
closely related to Legendrian curves in $\C^3$ 
(cf.\ Bryant \cite{B} and also Ejiri-Takahashi \cite{ET}
for the corresponding $\SL(2,\C)$-case).
In this paper, we use the techniques develop by the authors
in \cite{MUY} to produce
a complete bounded Legendrian holomorphic immersion
\[
    F:\D_1\longrightarrow \C^3.
\]
Recall that $F$ is called {\it Legendrian\/} 
if the pull-back of the canonical contact form 
\begin{equation}\label{eq:contact-c}
    \Omega_{\C}:=dx_3+x_2dx_1
\end{equation}
by $F$ vanishes, where $(x_1,x_2,x_3)$ is the canonical
complex coordinate system of $\C^3$.
The existence of such an $F$ is non-trivial,
since the correspondences between null curves and Legendrian curves
given in \cite{B} and \cite{ET} seem not to preserve neither
boundedness nor completeness.
Also, the authors do not know
whether
the method in \cite{AL}
can be applied for
Legendrian holomorphic immersions
by using a suitable modification or not. 

As applications, we are able to construct the following new examples: 
\begin{enumerate}
\item a weakly complete bounded flat front
      in $H^3$ (Theorem~\ref{thm:flat-front}),
\item a weakly complete bounded flat front
      in the de Sitter 3-space $S^3_1$
      (Theorem~\ref{thm:de-sitter}),
\item a weakly complete bounded improper affine 
      front in $\R^3$
      (Theorem~\ref{thm:affine}).
\end{enumerate}
It should be remarked that there are no
compact flat fronts in $H^3$ and $S^3_1$
(resp. improper affine fronts in $\R^3$). 
See Remark \ref{rem:cpt}.
A holomorphic map $E\colon{}\D_1\to\SL(2,\C)$ is called 
{\em Legendrian\/}
if the pull-back $E^*\Omega_{\SL}$ vanishes on $\D_1$, 
where $\Omega_{\SL}$ is the complex contact form on $\SL(2,\C)$
defined as 
\begin{equation}\label{eq:contact-sl}
    \Omega_{\SL}:=x_{11}dx_{22}-x_{12}dx_{21}.
\end{equation}
Here, elements in $\SL(2,\C)$ are represented by matrices
$(x_{ij})_{i,j=1,2}$. 
A holomorphic immersion $E\colon{}\D_1\to\SL(2,\C)$
is said to be {\em complete\/} if 
the pull-back metric $E^*g_{\SL}$
of the canonical Hermitian metric 
$g_{\SL}$ on $\SL(2,\C)$
is complete, see \eqref{eq:induced-metric}.

To construct the bounded holomorphic immersion 
$F:\D_1\to \C^3$, we show the following
\begin{maintheorem*}
 There exists a complete holomorphic Legendrian immersion 
 of the unit disk  $\D_1\subset\C$ into $\SL(2,\C)$ 
 such that its image is contained an arbitrary bounded domain
 in $\SL(2,\C)$.
\end{maintheorem*}
By Darboux's theorem, the contact structure of $\SL(2,\C)$
is locally Legendrian equivalent to that of $\C^3$.
Moreover, the following explicit transformation
 \[
    F\colon{}\C^3\ni (x,y,z)\longmapsto
            \begin{pmatrix}
	     e^{-z} & x e^{-z}\\
	     y e^z & e^z(1+xy)
	    \end{pmatrix}\in\SL(2,\C)
 \]
maps holomorphic contact curves in $\C^3$
to those in $\SL(2,\C)$.
Then if one take a complete Legendrian immersion of $\D_1$ 
into $\SL(2,\C)$ with sufficiently small image in $\SL(2,\C)$,
a Legendrian immersion into $\C^3$ is obtained.
Completeness follows from the same argument as 
\cite[Lemma 3.1]{MUY}.
In fact, since the image is bounded, the metrics
induced from $\SL(2,\C)$ and $\C^3$ are equivalent.

The paper is organized as follows:
In Section~\ref{sec:mainlemma}, we establish our formulations
and state the key-lemma to prove the main theorem,
which is proved in Section \ref{sec:proof}.
In Section \ref{sec:appl}, we give the applications as above.
In the appendix, we prepare a Runge-type theorem for Legendrian curves
in $\SL(2,\C)$ which is needed in Section~\ref{sec:proof}.

Finally, we mention the corresponding real problem,
that is, the existence of complete bounded  
Legendrian submanifolds immersed in $\R^{2n+1}$.
When $n=1$, there exists a
closed Legendrian curve immersed in 
an arbitrarily given open subset in $\R^3$: 
In fact, in \cite[Section 2]{KitU}, it is shown the existence of
a Legendrian curve 
contained in an arbitrary given open ball
of $P^3=T_1S^2$
(i.e. the unit cotangent bundle of $2$-sphere)
as a lift of an eye-figure curve.
Since any contact structure is locally rigid,
it gives an existence of
a closed Legendrian curve immersed in 
any ball of $\R^3$. 
Also, as an application of our construction,
we can construct a complete bounded Legendrian immersion
$L:\D_1\to B(\subset \R^5)$:
There exists a canonical projection (cf. \cite[Page 159]{KUY})
\[
   \pi:\SL(2,\C)\longrightarrow T_1^*H^3,
\]
where $T_1^*H^3$ is a unit cotangent bundle
of the hyperbolic 3-space $H^3$.
Then the projection of
our complete bounded holomorphic Legendrian
curve gives a
complete bounded Legendrian submanifold
immersed in 
an arbitrarily given open subset of
$T_1^*H^3$.
By Darboux's rigidity theorem,
this implies the existence of
a complete Legendrian immersion
$L:\D_1\to B$, 
where $B$ is an arbitrary ball
in $\R^5$.

\section{The Main Lemma}\label{sec:mainlemma}
In this section, we state the main lemma, which is 
an analogue of \cite[Main Lemma in page 121]{MUY}.
The main theorem in the introduction can be obtained as 
a direct conclusion of the main lemma in the same way
as in \cite{MUY}.
\subsection{Preliminaries}
We denote
$\imag=\sqrt{-1}$ and
\[
\D_r:=\{z\in\C\,;\,|z|<r\},\qquad
    \overline{\D}_r:=
    \D_r:=\{z\in\C\,;\,|z|\leq r\}
\]
for a positive number $r$.
Throughout this paper, 
the prime ${}'$ means the derivative with respect to 
the complex coordinate $z$ on $\C$.
\begin{proposition}\label{prop:weierstrass}
 A holomorphic immersion $X\colon{}\overline{\D}_1\to \SL(2,\C)$ 
 is Legendrian if and only if 
 $X^{-1}X'$ is anti-diagonal{\rm;}
 \begin{equation}\label{eq:weierstrass-data}
		\psi_X\,dz:=
		X^{-1}dX = \begin{pmatrix}
			      0 & \theta \\
			     \omega & 0 
			    \end{pmatrix}
                         = \frac{1}{\sqrt{2}}
			   \begin{pmatrix}
			     0 & \varphi_1 + \imag \varphi_2 \\
			     \varphi_1-\imag\varphi_2 & 0
			   \end{pmatrix}\,dz,
 \end{equation}
 where $\varphi_1$ and $\varphi_2$ are holomorphic
 functions on $\overline{\D}_1$.
 The metric 
induced by $X$ from the canonical Hermitian
 metric $g_{\SL}$
of $\SL(2,\C)$ is represented as
 \begin{equation}\label{eq:induced-metric}
    ds^2_X:=|\omega|^2+|\theta|^2=
          \bigl(|\varphi_1|^2+|\varphi_2|^2\bigr)|dz|^2.
 \end{equation}
 In particular, $\varphi_1$ and $\varphi_2$ have no 
 common zeros on $\overline{\D}_r$.
\end{proposition}
The holomorphic $1$-forms $\omega$ and $\theta$ in
\eqref{eq:weierstrass-data} are called the {\em canonical one forms\/}
for the flat front corresponding to $X$, see \cite{KRUY1}.
\begin{definition}\label{def:l-data}
 A pair of holomorphic functions $\varphi=(\varphi_1,\varphi_2)$
 on $\overline{\D}_1$ is called  {\em non-degenerate\/}
 if $\varphi_1$ and $\varphi_2$ have no common zeroes.
 The  pair $(\varphi_1,\varphi_2)$ 
 given by \eqref{eq:weierstrass-data} is called
 the {\em holomorphic data} of $X$.
 The matrix valued function
 \begin{equation}\label{eq:mat-form}
    M_\varphi:=\frac{1}{\sqrt{2}}
        \begin{pmatrix}
	  0 & \varphi_1+\imag\varphi_2 \\
	 \varphi_1 -\imag\varphi_2 & 0
	\end{pmatrix}
 \end{equation}
 is called {\em the matrix form\/} of the pair $\varphi$.
\end{definition}

\subsection{The Main Lemma}
To state the lemma,
we define the matrix norm $|A|$ of a $2\times 2$-matrix $A$ as
\begin{equation}\label{eq:mat-norm}
    |A|:=\sqrt{\trace(AA^*)}
        = \sqrt{\sum_{i,j=1,2}|A_{ij}|^2}
	\qquad \bigl(A=(A_{ij})_{i,j=1,2}\bigr).
\end{equation}
Note that if $A\in\SL(2,\C)$, 
then $|A|\geq \sqrt{2}$ holds.
The equality holds if and only if $A$ is the identity matrix.

For a vector $\vect{v}=(v_1,v_2)\in\C^2$, we set
$|\vect{v}|=\sqrt{|v_1|^2+|v_2|^2}$.
\begin{mainlemma*}\label{lem:main}
Let $X\colon{}\overline{\D}_1\to \SL(2,\C)$
be a holomorphic Legendrian immersion
 $X\colon{}\overline{\D}_1\to \SL(2,\C)$
 satisfies the following properties{\rm:}
\begingroup
 \renewcommand{\theenumi}{{\rm(\arabic{enumi})}}
 \renewcommand{\labelenumi}{{\rm(\arabic{enumi})}}
 \begin{enumerate}
 \item\label{assumption:1}
       $X(0)=\id$, where $\id$ is the identity matrix.
 \item\label{assumption:2}
$\overline{\D}_1$ contains the 
geodesic disc 
of radius $\rho$
centered at the origin
with respect to the induced metric
$ds^2_X$. 
 \item\label{assumption:3}
       There exists a number  $\tau>\sqrt{2}$ such that
       $|X|\leq \tau$ holds on  $\overline{\D}_1$.
 \end{enumerate}
\endgroup
 Then, for any positive numbers $\varepsilon$ and $s$, 
 there exists a holomorphic Legendrian immersion
 $Y\colon{}\overline{\D}_1\to \SL(2,\C)$
 such that
\begingroup
 \renewcommand{\theenumi}{{\rm(\roman{enumi})}}
 \renewcommand{\labelenumi}{{\rm(\roman{enumi})}}
 \begin{enumerate}
 \item\label{conclusion:1}
       $Y(0)=\id$,
 \item\label{conclusion:2}
$\overline{\D}_1$ contains the 
geodesic disc of radius $\rho+s$
centered at the origin
with respect to the induced metric
$ds^2_Y$, 
 \item\label{conclusion:3}
       $|Y|\leq \tau\sqrt{1+32s^2+\varepsilon}$ in  $\overline{\D}_1$,
 \item\label{conclusion:4}
       $|Y-X|<\varepsilon$ and
       $|\varphi_Y-\varphi_X|<\varepsilon$ in $\D_{1-\varepsilon}$,
       where $\varphi_X$ and  $\varphi_Y$ denote 
       holomorphic data of $X$ and $Y$,
       respectively.
 \end{enumerate}
\endgroup
\end{mainlemma*}
The main theorem in the introduction is obtained
by the same argument as \cite[Section 3.4]{MUY}).

\subsection{Key Lemma}
Now we state the key lemma, as an analogue of 
\cite[Key Lemma in page 129]{MUY}.
The main lemma in the previous subsection can be obtained
directly form the key lemma.

We work on the Nadirashvili's labyrinth \cite{Nadi}. 
Let us give  a brief description of this labyrinth:
Let $N$ be a (sufficiently large) positive number.
For $k=0,1,2,\dots,2N^2$, we set
\begin{equation}\label{eq:rj}
    r_k = 1-\frac{k}{N^3}
    \qquad
    \left(
      r_0=1, r_1=1-\frac{1}{N^3},\dots,r_{2N^2}=1-\frac{2}{N}
    \right),
\end{equation}
and let 
\begin{equation}\label{eq:dj}
    \D_{r_k}=\{z\in\C\,;\,|z|<r_k\} \quad \text{and}\qquad
    S_{r_k}=\partial \D_{r_k}
               =\{z\in\C\,;\,|z|=r_k\}.
\end{equation}
We define an annular domain $\A$ as
\begin{equation}\label{eq:all-annulus}
 \A := \D_1\setminus \D_{r_{2N^2}}=\D_1\setminus \D_{1-\frac{2}{N}},
\end{equation}
and 
\begin{gather*}
  A := \bigcup_{k=0}^{N^2-1} \D_{r_{2k}}\setminus \D_{r_{2k+1}},\quad
  \widetilde A := 
    \bigcup_{k=0}^{N^2-1} \D_{r_{2k+1}}\setminus  \D_{r_{2k+2}},\\
  L = \bigcup _{k=0}^{N-1}l_{\frac{2k\pi}{N}},\qquad
  \widetilde L = \bigcup _{k=0}^{N-1}l_{\frac{(2k+1)\pi}{N}},
\end{gather*}
where  $l_{\theta}$ is the ray
$l_{\theta}=\{re^{\imag\theta}\,;\,r\geq 0\}$.
Let $\Sigma$ be a compact set defined as 
\[
  \Sigma:= L \cup \widetilde L \cup S,\qquad
  S = \bigcup_{j=0}^{2N^2} \partial \D_{r_j}
    = \bigcup_{j=0}^{2N^2} S_{r_j},
\]
and define a compact set $\Omega$ by
\[
 \Omega = \A \setminus U_{{1}/{(4N^3)}}(\Sigma),
\]
where $U_{\varepsilon}(\Sigma)$ denotes 
the $\varepsilon$-neighborhood
(of the Euclidean plane $\R^2=\C$) of $\Sigma$.
Each connected component of $\Omega$ has width $1/(2N^{3})$.
For each number $j=1,\dots,2N$, we set
\begin{align*}
 \omega_j&:= 
   \bigl(l_{\frac{j\pi}{N}}\cap \A\bigr)
    \cup 
   \bigl(\text{
     connected components of $\Omega$
     which intersect with $l_{\frac{j\pi}{N}}$}\bigr){,}\\
 \varpi_j&:=U_{{1}/{(4N^3)}}(\omega_j) .
\end{align*}
Then $\omega_j$'s are compact sets.

\begin{keylemma*}
 Assume that a holomorphic Legendrian immersion
 $\L=\L_0\colon{}\overline{\D}_1\to\SL(2,\C)$
 satisfies{\rm:}
\begingroup
 \renewcommand{\theenumi}{{\rm(A-\arabic{enumi})}}
 \renewcommand{\labelenumi}{{\rm(A-\arabic{enumi})}}
 \begin{enumerate}
  \item\label{ass:1}
         $\L(0)=\id$,
  \item\label{ass:2}
$\overline{\D}_1$ contains the geodesic disc 
of radius $\rho$ centered at the origin
with respect to the metric $ds^2_{\L}$.
 \end{enumerate}
\endgroup
 Then for any positive number $\varepsilon$ and positive 
 number $s\in (0,{1}/{3})$, 
 there exists a sufficiently large integer $N$ and
 a sequence of holomorphic Legendrian immersions
 $\L_0=\L$, $\L_1$, \dots, $\L_{2N}$ of $\overline\D_1$ such that
\begingroup
 \renewcommand{\theenumi}{{\rm(C-\arabic{enumi})}}
 \renewcommand{\labelenumi}{{\rm(C-\arabic{enumi})}}
 \begin{enumerate}
  \item\label{concl:1}
         $\L_j(0)=\id$ {\rm(}$j=0,\dots,2N${\rm)},
  \item\label{concl:2}
         for each $j=1,\dots,2N$,
         $|\varphi_j-\varphi_{j-1}|<\varepsilon/(2N^2)$
         holds on $\D_1\setminus\varpi_j$,
         where $\varphi_j$ is the non-degenerate holomorphic data of $\L_j$,
  \item\label{concl:3}
        for each $j=1,\dots,2N$,
        \[
	    |\varphi_j|\geq \begin{cases}
			  c N^{3.5}\qquad &\text{on $\omega_j$}\\
			  c N^{-0.5}\qquad &\text{on $\varpi_j$}
			 \end{cases}
	\]
        holds, where  $c$ is a positive constant depending only on
        $\L=\L_0$,
  \item\label{concl:4}
$\overline{\D}_1$ contains the geodesic disc 
of radius $\rho+s$ centered at the origin
with respect to the metric $ds^2_{\L_{2N}}$,
  \item\label{concl:5}
       on  $\overline{\DD}_g$ as in \ref{concl:4},
       it holds that
       \[
	  |\L_{2N}|\leq 
          \left(\max_{\overline{\D}_1}|\L_0|\right)
           \sqrt{1+32s^2+({b}/{\sqrt{N}})},
       \]
       where $b$ is a positive constant depending only on  $\L=\L_0$.
 \end{enumerate}
\endgroup
\end{keylemma*}
The proof is given in  Section~\ref{sec:proof}.

\section{Proof of the Key Lemma}\label{sec:proof}
\subsection{Flat fronts in hyperbolic $3$-space}
We denote by $H^3$ the hyperbolic $3$-space, that is, 
the connected and simply connected $3$-dimensional space form of
constant sectional curvature $-1$, which is represented as
\begin{align}\label{eq:hyperbolic}
   H^3 &= \SL(2,\C)/\SU(2)
       = \{aa^*\,;\,a\in\SL(2,\C)\}\\
       &= \{X\in\Herm(2)\,;\,\det X=1, \trace X>0\},
       \qquad (a^*=\trans{\bar a}).\nonumber
\end{align}
where $\Herm(2)$ is the set of $2\times 2$ Hermitian matrices.
Identifying $\Herm(2)$ with the Lorentz-Minkowski $4$-space 
$\R^4_1$ as
\begin{equation}\label{eq:herm}
    \Herm(2)\ni
    \begin{pmatrix}
     x_0 + x_3 & x_1+\imag x_2 \\
     x_1-\imag x_2 & x_0-x_3
    \end{pmatrix}
    \quad
    \longleftrightarrow
    \quad
    (x_0,x_1,x_2,x_3)\in\R^4_1,
\end{equation}
the hyperbolic space $H^3$ can be considered as the 
connected component of the two-sheeted hyperboloid
\[
   \{(x_0,x_1,x_2,x_3)\in\R^4_1\,;\,
      -(x_0)^2+(x_1)^2+(x_2)^2+(x_3)^2, x_0>0\}.
\]
A Legendrian immersion $\L\colon{}\overline{\D}_1\to \SL(2,\C)$
induces a flat front
\[
   l=\L\L^*\colon{}\overline{\D}_1\longrightarrow H^3.
\]
Here flat fronts in $H^3$ are flat surfaces
with certain kind of singularities, see Section~\ref{sub:flat-front}.
The pull-back of the metric of $H^3$ by $l$ is computed as
\begin{equation}\label{eq:hyp-metric}
    ds^2_l:= |\omega|^2+|\theta|^2 +\omega\theta+\bar\omega\bar\theta 
           = |\omega+\bar\theta|^2
    \qquad
    \left(
      \L^{-1}d\L = \begin{pmatrix}
		     0 & \theta \\
		    \omega & 0
		   \end{pmatrix}
    \right),
\end{equation}
which is positive semi-definite and may degenerate.
On the other hand, let $ds^2_{\L}$ be the pull-back of the canonical
Hermitian metric of $\SL(2,\C)$ by $\L$.
Then by \eqref{eq:induced-metric}, we have
\[
    ds^2_{\L}-\frac{1}{2}ds^2_l
    =\frac{1}{2}\bigl(
                 |\omega|^2+|\theta|^2 -\omega\theta-\bar\omega\bar\theta 
                \bigr)
    =\frac{1}{2}|\omega-\bar\theta|^2\geq 0,
\]
and hence 
\begin{equation}\label{eq:metric-estimate}
      ds^2_l\leq 2ds^2_{\L}
\end{equation}
holds.
For any path $\gamma$ in $\overline{\D}_1$ joining 
$x$ and $y\in \D_1$, it holds that
\begin{equation}\label{eq:length-estimate}
 \Length_{ds^2_l}\gamma:=
  \int_{\gamma}ds_l \geq \dist_{H^3}\bigl(l(x),l(y)\bigr),
\end{equation}
where $\dist_{H^3}$ denotes the distance in the hyperbolic $3$-space.
On the other hand, 
\begin{equation}\label{eq:dist}
 2\cosh \dist_{H^3}\bigl(o,l(x)\bigr) =|\L(x)|^2
\end{equation}
holds (see \cite[Lemma A.2]{MUY}), where
we set
\[
     o:=\id=
	         \begin{pmatrix}
		  1 & 0 \\
		  0 & 1
		 \end{pmatrix},
\]
namely $o$ is the point on $H^3$ which 
corresponds to the origin of the Poincar\'e ball
model.

\subsection{Inductive construction of  $\L_j$'s}
In this section, we describe the recipe to construct 
a sequence $\L_0$, $\L_1$, \dots, $\L_{2N}$ in 
Key Lemma.
Assume 
$\L_0$, \dots, $\L_{j-1}$ are already obtained, and we 
shall now construct $\L_j$ as follows:
Let
\begin{equation}\label{eq:zeta-j}
   \zeta_j:=\left(
          1-\frac{2}{N}-\frac{4}{N^3}
         \right)e^{\imag\pi j/N}
\end{equation}
be the base point of the compact set $\omega_j$
given in \cite[Fig.\ 1]{MUY}.
We set
\[
    E_0(z):=\L_{j-1}(\zeta_j)^{-1}\,\L_{j-1}(z),
    \qquad
    f_0(z) := E_0(z)E_0^*(z).
\]
That is, $E_0$ is  the Legendrian immersion 
with the same holomorphic data as $\L_{j-1}$ 
such that $E_0(\zeta_j)=\id$,
and $f_0$ the corresponding flat front.
Here, if we write 
\begin{equation}\label{eq:initial}
    f_0(0) = \begin{pmatrix}
	    \xi_0 + \xi_3 & \xi_1 + \imag \xi_2\\
	    \xi_1-\imag \xi_2 & \xi_0-\xi_3
	     \end{pmatrix}\in H^3,
\end{equation}
 then there exist real numbers $t$ and $\hat \xi_1$
 such that
 \begin{equation}\label{eq:matrix-a}
     af_0(0)a^* = 
     \begin{pmatrix}
	    \xi_0 + \xi_3 & \hat \xi_1\\
	    \hat \xi_1 & \xi_0-\xi_3
     \end{pmatrix}\in H^3
   \qquad
   \left(
     a=\begin{pmatrix}
	e^{\imag t} & 0 \\
	0 & e^{-\imag t}
       	\end{pmatrix}\in \SL(2,\C)\right).
 \end{equation}
 We now set
 \begin{equation}\label{eq:rotate}
   E := a E_0 a^*, \qquad
   \psi := E^{-1}E'.
 \end{equation}
 Then it holds that
 \[
   \psi=a^* \psi_{j-1} a \qquad (\psi_{j-1}:=\L_{j-1}^{-1}\L_{j-1}').
 \]
 Let 
 $\varphi=(\varphi_1,\varphi_2)$ be the
 holomorphic data induced from $E$
 such that $\psi=M_\varphi$ (see \eqref{eq:mat-form}).
 Applying Lemma~\ref{lem:runge} in the appendix to 
 this $\varphi$,
 we get
 a new holomorphic data $\tilde\varphi$.
 Take $\tilde E$ such that
 \begin{equation}\label{eq:def-ode}
     \tilde E^{-1}\tilde E'=\tilde\psi,\qquad
     \tilde E(\zeta_j)=\id
     \qquad(\tilde\psi=M_{\tilde\varphi}).
 \end{equation}

 Finally, we set
 \begin{equation}\label{eq:L-j}
   \L_j := a^*\{\tilde E(0)\}^{-1} \tilde E a,
 \end{equation}
 which gives the desired Legendrian immersion.

\subsection{Estimation of interior distance}
Applying the inductive constriction,
we get a sequence of
Legendrian immersion $\L_{1},\cdots,\L_{2N}$.
Then one can prove
 \ref{concl:1},
 \ref{concl:2}, \ref{concl:3} and \ref{concl:4}
by the exactly same argument as in
 \cite[Page 129]{MUY}.
 In fact, although the data $\varphi$ (a pair of holomorphic functions)
 is different from that in \cite{MUY} (a triple of 
 holomorphic functions),
the proof of the key lemma
(as in \cite[Page 129]{MUY}) 
only needs the norm $|\varphi|$ of the data $\varphi$.
In \cite{MUY}, we were working on the metric
$ds^2_{\mathcal B_j}=|\psi_j|^2\, dz\,d\bar z$.
In this paper, we now use the metric
$ds^2_{\L_j}=|\psi_j|^2\, dz\,d\bar z$,
where $\psi_j$ is given in \eqref{eq:rotate}.
If one replaces $|\psi_j|$ in \cite{MUY} by
this new $|\psi_j|$ as above,
then the completely same argument as \cite{MUY} works 
in our case.

\subsection{Extrinsic distance}
Thus, the only remaining assertion we should prove
is \ref{concl:5} of Key Lemma. We shall now prove it:
By \ref{concl:4}, $(\overline{\D_1},ds^2_{\L_{2N}})$ contains 
a geodesic disc $\overline{\DD}_g$ centered at origin with radius  $\rho+s$.
By the maximum principle,
it is sufficient to show that for  each $p\in \partial\DD_g$,
 it holds that
\begin{equation}\label{eq:concl}
   |\L_{2N}(p)|\leq 
          \left(\max_{\overline{\D}_1}|\L_0|\right)
           \sqrt{1+32s^2+({b}/{\sqrt{N}})}
	    \qquad (p\in \partial {\DD}_g),
\end{equation}
 where $b$ is a positive constant depending only on the initial
 immersion $\L_0$.

 If $p\in \partial \DD_g$ is not in  $\varpi_1\cup\dots \cup \varpi_{2N}$,
 the same argument as \cite[Page 129]{MUY} implies the conclusion.
 So, it is sufficient to consider the
case that there exists $j\in\{1,...,2N\}$
such that
\begin{equation}\label{eq:boundary}
   p \in \partial \DD_g \cap \varpi_j.
\end{equation}
We fix such a $p$. 
{\it From now on, the symbols $c_k$ $(k=1,2,\dots)$
 denote  suitable positive constants,
 which depend only on the initial data $\L=\L_0$.}

Like as in the proof of the inequality
\cite[(4.8)]{MUY}, we can apply
\cite[Corollary A.6]{MUY} 
for $X:=\L_j$ and $Y:=\L_{j-1}$.
Then  we have
 \begin{equation}\label{eq:safe-est}
    \dist_{H^3}\bigl(l_j(z),l_{j-1}(z)\bigr)\leq
    \frac{c_1\varepsilon}{2N^2}
    \qquad \text{(on $\D_1\setminus\varpi_j$)},
 \end{equation}
where $l_j:=\L_j\L_j^*$ and $l_{j-1}:=\L_{j-1}\L_{j-1}^*$.
Taking \eqref{eq:safe-est} into account,
inequality \eqref{eq:concl}
reduces to the following inequality
\begin{equation}\label{eq:concl-2}
   |\L_{j}(p)|\leq 
          \left(\max_{\overline{\D}_1}|\L_0|\right)
           \sqrt{1+32s^2+({c_2}/{\sqrt{N}})}
	    \qquad (p\in \partial \DD_g\cap \varpi_j).
\end{equation}

 Take the $ds^2_{\L_{2N}}$-geodesic $\gamma_0$ joining 
 $p$ and the origin $0\in\overline{\D}_1$ and 
 denote by $\hat p$ the first point on $\gamma_0$ which 
 meets $\partial\varpi_j$.
 Then we have
 \begin{equation}\label{eq:s-est}
   \dist_{ds^2_{\L_{2N}}} (\hat p,p)\leq s + \frac{c_3}{\sqrt{N}}.
 \end{equation}
 In fact, let $\gamma_1$ (resp.\ $\gamma_2$) be the subarc
 of $\gamma_0$ which joins $0$ and $\hat p$ (resp.\ $\hat p$ and $p$).
 Since $\gamma_0$ is the geodesic and $\DD_g$ is the disc 
 with radius $\rho+s$, \ref{concl:2} and \ref{concl:3} yields
 that
 \[
    \rho+s = \int_{\gamma_0} ds_{\L_{2N}}
           = \int_{\gamma_0}|\varphi_{2N}|\,|dz|
	   \geq  \frac{c_4}{\sqrt{N}}L_{\gamma_0},
 \]
 where $L_{\gamma_0}$ is the length of $\gamma_0$ with respect to the 
 Euclidean metric of $\C$.
 Thus, we have $L_{\gamma_0} \leq  c_6\sqrt{N}(\rho+s)$.
 On the other hand, 
 let $\sigma$ be the shortest
 line segment on $\C$ joining $\hat p$ and $\partial\D_1$.
 Then the Euclidean length $L_{\sigma}$ of $\sigma$ 
 satisfies $L_{\sigma}\leq c_7/N$.
 Thus, by \ref{ass:2}, we have
 \begin{align*}
    \rho+s &= \int_{\gamma_0}|\varphi_{2N}|\,|dz|
            \geq  \int_{\gamma_1}|\varphi_{0}|\,|dz|
              -\int_{\gamma_1}
                |\varphi_{2N}-\varphi_{0}|\,|dz|
              +\int_{\gamma_2}|\varphi_{2N}|\,|dz|\\
        &\geq \int_{\gamma_1\cup\sigma}
               |\varphi_0|\,|dz|
           -\int_{\sigma}
               |\varphi_0|\,|dz| - 
            L_{\gamma_0}\frac{\varepsilon}{N}
            +\dist_{ds^2{\L_{2N}}}(\hat p,p)\\
        &\geq \rho - \frac{c_3}{\sqrt{N}} + 
              \dist_{ds^2{\L_{2N}}}(\hat  p,p)
 \end{align*}
 which implies \eqref{eq:s-est}.

 Moreover, 
 since $\zeta_j$ and $\hat p$ can be joined 
 by a path $\gamma$ on 
 $\overline{\D_1}\setminus(\varpi_1\cup\dots\cup\varpi_{2N})$
 whose Euclidean length is not greater than $c_8/N$
 (see \cite[Fig.\ 1]{MUY}),
 we have
\begin{equation}\label{eq:boundary-distance}
   \dist_{ds^2_{\L_{2N}}} (\zeta_j,p)\leq s + \frac{c_9}{\sqrt{N}}
\end{equation}
 (see \cite[(4.9)]{MUY}).
 
\begin{lemma}\label{lem:est1}
 In the above setting, the following inequalities hold{\rm:}
 \begin{align}
    &|\L_{j-1}(\zeta_j)|^2\leq 
    \left(\max_{z\in\overline{\D}_1}|\L_0(z)|^2
    \right)\left(1+\frac{c_{10}}{N}\right),\label{est:1a}\\ 
    &\dist_{H^3}\bigl(l_{j-1}(\zeta_j),l_{j}(\zeta_j)\bigr)\leq
      \frac{c_{11}}{N^2} 
      \label{est:1b}.
 \end{align}
\end{lemma}
\begin{proof}
 The first inequality holds because
 \[
   |\L_{j-1}(\zeta_j)|^2=2\cosh\bigl(\dist_{H^3}(o,l_{j-1}(\zeta_j)\bigr)
    \leq 2\cosh\left(\dist_{H^3}\bigl(o,l_{0}(\zeta_j)\bigr)
               +\frac{c_{12}}{N}\right), 
 \]
 see \eqref{eq:dist}.
 The second inequality directly follows from
 \eqref{eq:safe-est}.
\end{proof}

Now, recall the procedure constructing  $\L_{j}$ from $\L_{j-1}$:
Two Legendrian immersions $E$, $\tilde E$ 
in \eqref{eq:rotate} and \eqref{eq:def-ode}
are congruent to  $\L_{j-1}$, $\L_j$,
respectively, and satisfy $E(\zeta_j)=\tilde E(\zeta_j)=\id$.
Take flat fronts $f:=EE^*$ and $\tilde f=\tilde E\tilde E^*$ 
associated to $E$ and $\tilde E$, respectively.
By a choice of the matrix $a$ in \eqref{eq:matrix-a},
the points $f(\zeta_j)=o(=\id)$ and $f(0)$ lie on the 
``$x_1x_3$-plane'' $\Pi$, i.e.
\[
   \Pi:=\left\{
         \begin{pmatrix}
	  x_0+x_3 & x_1\\
	  x_1 & x_0-x_3
	 \end{pmatrix}\in H^3\,;\,x_0,x_1,x_3\in\R
        \right\}.
\]
Let $\Lambda$ be the geodesic of $H^3$ passing through the 
origin  $o$ ($=\id$) and perpendicular to $\Pi$
(i.e.,  the $x_2$-axis), and let
$q$ be the foot of the perpendicular from $\tilde f(p)$ to 
 the line $\Lambda$.
\begin{lemma}\label{lem:est2}
 In the above circumstances, one has{\rm:}
 \[
    \dist_{H^3}\bigl(o,q\bigr)\leq 2s + \frac{c_{13}}{\sqrt{N}}. 
 \]
\end{lemma}
\begin{proof}
 The triangle $\triangle oq\tilde f(p)$ is a right triangle such that 
 the angle $q$ is the right angle.
 Then by  \eqref{est:1b}, \eqref{eq:boundary-distance} and
 \eqref{eq:metric-estimate}, we have
\allowdisplaybreaks{%
 \begin{align*}
  \dist_{H^3}\bigl(o,q\bigr)&\leq 
  \dist_{H^3}\bigl(o,\tilde f(p)\bigr)=
  \dist_{H^3}\bigl(l_{j-1}(\zeta_j),l_j(p)\bigr)\\
  &\leq 
  \dist_{H^3}\bigl(l_{j-1}(\zeta_j),l_j(\zeta_j)\bigr)+
  \dist_{H^3}\bigl(l_{j}(\zeta_j),l_j(p)\bigr)
  \leq \frac{c_{11}}{N}+\dist_{ds^2_{l_j}}(\zeta_j,p)\\
  &\leq \frac{c_{11}}{N}+2\dist_{ds^2_{\L_j}}(\zeta_j,p)
   \leq \frac{c_{14}}{N}+2\dist_{ds^2_{\L_{2N}}}(\zeta_j,p)\\
  & \leq \frac{c_{15}}{N}+2\left(s+\frac{c_8}{\sqrt{N}}\right)
  \leq 2s + \frac{c_{13}}{\sqrt{N}}.
 \end{align*}%
}
 Thus we have the conclusion.
\end{proof}

\begin{lemma}\label{lem:length-est}
 Under the hypotheses above, one has{\rm:}
 \[
    \dist_{H^3}\bigl(q,\tilde f(p)\bigr)\leq 14s^2 + 
       \frac{c_{10}}{\sqrt{N}}.
 \]
\end{lemma}
\begin{proof}
 Let $\varphi=(\varphi_1,\varphi_2)$ and 
 $\tilde\varphi=(\tilde\varphi_1,\tilde\varphi_2)$ 
 be the holomorphic data of the Legendrian immersions
 $E$ and $\tilde E$, respectively, that is,
 \begin{align*}
    \psi &:=E^{-1}E' =\frac{1}{\sqrt{2}}\begin{pmatrix}
			       0 & \varphi_1 + \imag \varphi_2\\
			      \varphi_1 - \imag \varphi_2 & 0
			     \end{pmatrix}=M_{\varphi},\\
    \tilde\psi &:=\tilde E^{-1}\tilde E= \frac{1}{\sqrt{2}}\begin{pmatrix}
			       0 & \tilde\varphi_1 + \imag \tilde\varphi_2\\
			      \tilde\varphi_1 - \imag \tilde\varphi_2 & 0
			     \end{pmatrix}=M_{\tilde\varphi}
 \end{align*}
 hold.
 Set
 \[
     F(z):= \int_{\zeta_j}^z \psi(z)\,dz,\qquad\text{and}\qquad
    \tilde F(z):= \int_{\zeta_j}^z \tilde\psi(z)\,dz.
 \]
We define two values $\Delta$ and $\tilde \Delta$
by 
 \[
    E(p) = \id + F(p)+\Delta,\qquad
    \tilde E(p) = \id + \tilde F(p)+\tilde \Delta.
 \]
Since  $E(\zeta_j)=\tilde E(\zeta_j)=\id$,
\cite[Appendix A.4]{MUY} yields that
 \[
    |\Delta|\leq 
    \left[
    \bigl(
     \max_{\gamma}|E|
    \bigr)
    \int_{\gamma}|\psi|\,|dz|
    \right]^2
     ,\quad
    |\tilde\Delta|\leq 
    \left[
    \bigl(
     \max_{\gamma}|\tilde E|
    \bigr)
    \int_{\gamma}|\tilde \psi|\,|dz|
    \right]^2.
 \]
 Here $\gamma$ is a path joining $\zeta_j$ and $p$
 as in \cite[Fig. 1]{MUY}.
 This argument is completely parallel to
 that in \cite[Page 132]{MUY}.
 Since the Euclidean length of $\gamma$
 in $\D_1$ is bounded by  $c_{17}/N$, we have
 \[
    |F(p)|\leq \int_{\gamma}|\psi|\,|dz|\leq \frac{c_{18}}{N}.
 \]
 On the other hand, let $\tilde\gamma$
 be the $ds^2_{\L_{j}}$-geodesic joining $\zeta_j$ and $p$.
 Then noticing that $ds^2_{\L_{j}}=ds^2_{\tilde E}$,
 \eqref{eq:boundary-distance} implies that 
 \[
    |\tilde F(p)|\leq \int_{\gamma}|\tilde\psi|\,|dz|
        =\int_{\gamma}ds_{\tilde F}\leq s+\frac{c_{19}}{\sqrt{N}}.
 \]
 On the other hand,
since
\allowdisplaybreaks{%
 \begin{align*}
    \max_{\gamma}|E|^2 &= 
    \max_{\gamma} 
      \left\{2 \cosh \dist_{H^3}\bigl(o,f(z)\bigr)\right\}\\
   &=\max_{\gamma}
       \left\{
       2 \cosh \dist_{H^3}\bigl(l_{j-1}(\zeta_j),l_{j-1}(z)\bigr)
       \right\}
   \leq 2\cosh \frac{c_{20}}{N}
    \leq 2\left(1+\frac{c_{21}}{N}\right),\\
    \max_{\gamma}|\tilde E|^2&\leq 
         2 \cosh \dist_{H^3}\bigl(o,\tilde f(p)\bigr)
       =2 \cosh \dist_{H^3}\bigl(l_{j-1}(\zeta_j),l_{j}(p)\bigr)\\
       &\leq 2 \cosh
          \{\dist_{H^3}\bigl(l_{j-1}(\zeta_j),l_{j}(\zeta_j)\bigr)+
               \dist_{H^3}\bigl(l_{j-1}(\zeta_j),l_{j}(p)\bigr)\}\\
    &   \leq  2\cosh\left(2s+\frac{c_{22}}{\sqrt{N}}\right)
       \leq 2\left(1+4s^2+\frac{c_{23}}{\sqrt{N}}\right),
 \end{align*}%
}
 we have
 \[
     |\Delta|\leq \frac{c_{24}}{N},\qquad
     |\tilde\Delta|\leq 4s^2 + \frac{c_{25}}{\sqrt{N}},
 \]
 whenever $s<1/3$.
 Now, we set
 \begin{align*}
    \begin{pmatrix}
     x_0 + x_3 & x_1+\imag x_2 \\
     x_1-\imag x_2 & x_0-x_3
    \end{pmatrix}:&= \tilde f(p)
    = (\id + \tilde F  +\tilde\Delta)(\id+\tilde F^*+\tilde\Delta^*)\\
    &=
    \id + \tilde F +\tilde F^* - F-F^*+\delta.
 \end{align*}
 Then, reasoning as in \cite[Page 133]{MUY}, we have
 \[
    |\delta| = |F+F^*+\tilde F\tilde F^*
     +\Delta+\tilde\Delta+\tilde\Delta\tilde F^*
   + \tilde F\tilde\Delta^*+\tilde\Delta\tilde\Delta^*|
     < 14 s^2 + \frac{c_{26}}{\sqrt{N}}.
 \]
 Moreover, if we define
 \begin{align*}
  h(z) &:= F(z)+F^*(z) = 
          \sqrt{2}\int_{\zeta_j}^z 
                         \begin{pmatrix}
			      0 & \Re\varphi_1+ \imag\Re\varphi_2 \\
			       \Re\varphi_1- \imag\Re\varphi_2 & 0
                         \end{pmatrix},\\
  \tilde h(z) &:= \tilde F(z)+\tilde F^*(z) = 
          \sqrt{2}\int_{\zeta_j}^z 
                         \begin{pmatrix}
			      0 & \Re\tilde\varphi_1+ \imag\Re\tilde\varphi_2 \\
			       \Re\tilde\varphi_1- \imag\Re\tilde\varphi_2 & 0
                         \end{pmatrix},
 \end{align*}
 the $x_3$-component of  $\tilde h-h$ vanishes, and 
 the $x_1$-component is
 \[
    \sqrt{2}\int\Re(\tilde\varphi_1-\varphi_1)
    = |\tilde h(z)-h(z)|\frac{u_2}{u_1}
    \leq \left(s+\frac{c_{27}}{\sqrt{N}}\right)\frac{u_2}{u_1}
    \leq \frac{c_{28}}{N},
 \]
because of the property of $\vect{u}$ in Lemma~\ref{lem:runge}.
 Thus, we have
 \[
     |x_1|\leq 14s^2 + \frac{c_{29}}{\sqrt{N}},\qquad
     |x_3|\leq 14s^2 + \frac{c_{30}}{\sqrt{N}}.
 \]
Since $\dist_{H^3}\bigl(\tilde f(p),q\bigr)$ is
 the distance between  $\tilde f(p)$ and $x_2$-axis,
 we have
 \[
   \dist_{H^3}\bigl(\tilde f(p),q\bigr)
    = \sinh^{-1}\sqrt{x_1^2+x_3^2}
    \leq 14s^2 + \frac{c_{16}}{\sqrt{N}},
 \]
 which proves the conclusion.
\end{proof}

\subsection{Proof of Key Lemma}
Note that for any positive numbers $x$ and $y$, it holds that
\begin{multline}\label{eq:cosh-est}
  \cosh(x+y)= \cosh x \cosh y + \sinh x \sinh y\\
            \leq \cosh x(\cosh y + \sinh y)= e^y\cosh x.
\end{multline}

By \eqref{eq:safe-est}, we have
\begin{equation}\label{eq:dist0}
 \dist_{H^3}\bigl(o,l_{j-1}(\zeta_j)\bigr)
    \leq \dist_{H^3}\bigl(o,l_0(\zeta_j)\bigr)+\frac{c_{31}}{N^2}.
\end{equation}
Under the situation here,  $f(\zeta_j)$ and $o(=\id)$ lie on 
the $x_1x_3$-plane, and 
$q$ is on the $x_2$-axis, the geodesic triangle 
$\triangle f(\zeta_j) o q$ in $H^3$
is a right triangle.
Then by the hyperbolic Pythagorean theorem, we have
\allowdisplaybreaks{%
\begin{align*}
 \cosh\dist_{H^3}\bigl(f(0),q\bigr)&=
 \cosh\dist_{H^3}\bigl(f(0),o\bigr)
 \cosh\dist_{H^3}\bigl(o,q\bigr)\\
 &=
 \cosh\dist_{H^3}\bigl(f(0),f(\zeta_j)\bigr)
 \cosh\dist_{H^3}\bigl(o,q\bigr)\\
 &=
 \cosh\dist_{H^3}\bigl(l_{j-1}(0),l_{j-1}(\zeta_j)\bigr)
 \cosh\dist_{H^3}\bigl(o,q\bigr)\\
 &=
 \cosh\dist_{H^3}\bigl(o,l_{j-1}(\zeta_j)\bigr)
 \cosh\dist_{H^3}\bigl(o,q\bigr)\\
 &=
 \cosh\left(\dist_{H^3}\bigl(o,l_{0}(\zeta_j)\bigr)+\frac{c_{31}}{N^2}\right)
 \cosh\dist_{H^3}\bigl(o,q\bigr)\\
 &
 \leq \exp\left(\frac{c_{31}}{N^2}\right)
      \cosh\dist_{H^3}\bigl(o,l_{0}(\zeta_j)\bigr)
      \cosh\dist_{H^3}\bigl(o,q\bigr)\\
 &= 
  \frac{1}{2}|\L_0(\zeta_j)|^2
      \exp\left(\frac{c_{31}}{N^2}\right)
      \cosh\dist_{H^3}\bigl(o,q\bigr)\\
 & \leq 
    \left(\frac{1}{2}\max_{\D_1}|\L_0|^2\right)
      \cosh\left(2s+\frac{c_{13}}{\sqrt{N}}\right)\left(1+\frac{c_{32}}{N^2}\right)
 \\
 &\leq 
    \left(\frac{1}{2}\max_{\D_1}|\L_0|^2\right)
    \left(1+\frac{c_{32}}{N^2}\right)
    \left(1+4s^2+\frac{c_{33}}{\sqrt{N}}\right)\\
 &\leq 
    \left(\frac{1}{2}\max_{\D_1}|\L_0|^2\right)
     \left(1+4s^2 + \frac{c_{34}}{\sqrt{N}}\right).
\end{align*}%
}
Thus, by using  Lemmas \ref{lem:est1}, 
\ref{lem:est2}, and \ref{lem:length-est},
we have
\begin{align*}
 \frac{1}{2}|\L_j(p)|^2 
  &=
  \cosh\dist_{H^3}\bigl(o,l_j(p)\bigr)=
  \cosh\dist_{H^3}\bigl(\tilde f(0),\tilde f(p)\bigr)\\
  &\leq
  \cosh 
    \left(
    \dist_{H^3}\bigl(\tilde f(0),f(0)\bigr) +
    \dist_{H^3}\bigl(f(0),q\bigr) +
    \dist_{H^3}\bigl(q,\tilde f(p)\bigr) 
    \right)\\
  &\leq 
  \cosh 
    \left(
    \dist_{H^3}\bigl(l_j(\zeta_j),l_{j-1}(\zeta_j)\bigr) +
    \dist_{H^3}\bigl(f(0),q\bigr) +
    \dist_{H^3}\bigl(q,\tilde f(p)\bigr) 
    \right)\\
  &\leq 
  \cosh 
    \left(
     \frac{c_{11}}{N^2}+
    \dist_{H^3}\bigl(f(0),q\bigr) +
    \dist_{H^3}\bigl(q,\tilde f(p)\bigr) 
    \right)\\
   &\leq 
    \exp\left(\frac{c_{11}}{N^2}\right)
    \cosh\left(
    \dist_{H^3}\bigl(f(0),q\bigr) +
    \dist_{H^3}\bigl(q,\tilde f(p)\bigr) 
    \right)\\
   &\leq 
    \left(1+\frac{c_{35}}{N^2}\right)
    \cosh\left(
    \dist_{H^3}\bigl(f(0),q\bigr) +
    \dist_{H^3}\bigl(q,\tilde f(p)\bigr) 
    \right)\\
   &\leq 
    \left(1+\frac{c_{35}}{N^2}\right)
    \exp\left(\dist_{H^3}\bigl(q,\tilde f(p)\bigr)\right)
    \cosh    \dist_{H^3}\bigl(f(0),q\bigr) \\
   &\leq 
    \left(1+\frac{c_{35}}{N^2}\right)
    \exp\left(14s^2+\frac{c_{16}}{\sqrt{N}}\right)
    \cosh    \dist_{H^3}\bigl(f(0),q\bigr) \\
   &\leq
    \left(1+\frac{c_{35}}{N^2}\right)
    \left(1+2\left(14s^2+\frac{c_{16}}{\sqrt{N}}\right)\right)
    \left(\frac{1}{2}\max_{\D_1}|\L_0|^2\right)
     \left(1+4s^2 + \frac{c_{36}}{\sqrt{N}}\right)\\
   &\leq 
    \left(\frac{1}{2}\max_{\D_1}|\L_0|^2\right)
    \left(1+32s^2+\frac{c_2}{\sqrt{N}}\right)
\end{align*}
which proves the conclusion.

\section{Applications}\label{sec:appl}
This section is devoted to prove some applications of the main theorem
as we stated in the introduction.

A smooth map $f\colon{}\D_1 \to M^3$ 
into a $3$-manifold $M^3$
is called a ({\em wave}) {\em front\/}
if there exists a Legendrian immersion
$L_f\colon{}\D_1\to P(T^*M^3)$ with respect 
to the canonical contact structure of the projective
cotangent bundle $\pi:P(T^*M^3)\to M^3$ such 
that $\pi\circ L_f=f$.
\subsection{Flat fronts in hyperbolic $3$-space}
\label{sub:flat-front}
Recall that
\[
   H^3:=\SL(2,\C)/\SU(2)=
        \{a a^*\,;\,a\in\SL(2,\C)\}\qquad (a^*=\trans{\bar a}).
\]
For a Legendrian immersion $\L\colon{}\D_1\to\SL(2,\C)$,
the projection
\begin{equation}\label{eq:flat-front}
   f:=\L\L^*\colon{}\D_1\longrightarrow H^3
\end{equation}
gives a {\em flat front\/} in $H^3$
 (see \cite{KUY,KRUY1} for the definition of
flat fronts).
We call $\L$ in \eqref{eq:flat-front} the {\em holomorphic lift\/}
of $f$.
A flat front $f$ is called {\em weakly complete\/}
if its holomorphic lift is complete 
with respect to the induced metric $ds^2_{\L}$
\cite{KRUY2,UY}.

Let $f\colon{}\D_1\to H^3$ be a flat front and $\L$ its holomorphic lift.
Take $\omega$ and $\theta$ as in \eqref{eq:hyp-metric} and set
\begin{equation}\label{eq:rho}
    \rho:= \frac{\theta}{\omega}\colon{}\D_1\longrightarrow \C\cup\{\infty\}.
\end{equation}
Then a point $z\in \D_1$ is a singular point if and only if
$|\rho(z)|=1$.
We denote by $\Sigma_f$ the singular set of $f$;
\[
   \Sigma_f:=\{z\in\D_1\,;\,|\rho(z)|=1\}.
\]
We have the following
\begin{theorem}\label{thm:flat-front}
 There exists a weakly complete flat front $f\colon{}\D_1\to H^3$ 
 whose image is bounded in $H^3$
 such that $\D_1\setminus\Sigma_f$ is open dense in $\D_1$.
\end{theorem}

\begin{proof}
 Let
 $\L\colon{}\D_1\to\SL(2,\C)$ be a Legendrian immersion 
 as in the main theorem, and  set 
 $f:=\L\L^*\colon{}\D_1\to H^3$, which gives
a flat front.
The boundedness of $f$
follows from that of $\L$, and the weak completeness of $f$
 follows from the completeness of $\L$.

Finally, we shall prove that $\D_1\setminus\Sigma_f$
is open dense:
If $\Sigma_f$ has an interior point, 
then $\rho$ in \eqref{eq:rho} satisfies $|\rho|=1$
identically because of the analyticity of $\rho$.
However, if we take a initial immersion so that
$|\rho|$ is not constant, then
the resulting bounded weakly complete flat front
has the same property, since our iteration
can be taken to be small enough
near the origin of $\D_1$.
\end{proof}

\subsection{Flat fronts in de Sitter $3$-space}
\label{sub:de-sitter}
The de Sitter $3$-space $S^3_1$ is the connected and simply connected
Lorentzian space form of constant curvature $1$, which
is represented as
\[
   S^3_1 = \SL(2,\C)/\SU(1,1)
         = \{ae_3a^*\,;\,a\in\SL(2,\C)\}
	 \qquad
	 \left(e_3:=\begin{pmatrix}
		     1 & \hphantom{-}0\\
		     0 & -1\end{pmatrix}\right).
\]
Let $\L\colon{}\D_1\to \SL(2,\C)$ be a Legendrian immersion.
Then the projection
\[
   f\colon{}\L e_3\L^*\colon{}\D_1\longrightarrow S^3_1
\]
gives a flat front.
We remark that $f$ is not an immersion at $z$ 
if and only if $|\rho(z)|=1$,
where $\rho=\theta/\omega$ as in \eqref{eq:rho}.
The singular set $\Sigma_f$
of $f$ is characterized by $|\rho(z)|=1$.
Similar to the case of flat fronts in $H^3$,
$f$ is said to be {\em weakly complete\/} if the metric
induced from the canonical Hermitian metric of $\SL(2,\C)$
by $\L$ is complete.
Then we have
\begin{theorem}\label{thm:de-sitter}
 There exists a weakly complete flat front $f\colon{}\D_1\to H^3$ 
 whose image is bounded in $S^3_1$ such 
 that $\D_1\setminus\Sigma_f$ is open dense in $\D_1$, where $\Sigma_f$
 is the set of singular points of $f$.
\end{theorem}
\begin{proof}
 Let $\L\colon{}\D_1\to\SL(2,\C)$ be a Legendrian immersion
 as in the main theorem, and set $f:=\L e_3\L^*$.
 Then $f$ is a weakly complete flat front in $S^3_1$.
 Moreover, one can see that  $\D_1\setminus\Sigma_f$
 is open dense by the same argument as in Theorem~\ref{thm:flat-front}.
\end{proof}

See \cite{GMM2}, \cite{KU} and \cite{AE}
for the relationships between
flat surfaces 
and linear Weingarten surfaces
in $H^3$ or $S^3_1$.

\subsection{Improper Affine front in affine $3$-space}
\label{sub:affine}
A notion of {\em IA-maps\/}
in the affine $3$-space has been introduced 
by A. Mart\'{\i}nez in \cite{Martinez}.
IA-maps  are improper affine spheres
with a certain kind of singularities.
Since all of IA-maps
are wave fronts (see \cite{Nakajo, UY}),
we call them  {\em improper affine fronts}
(The terminology \lq improper affine fronts\rq\
 has been already used in
Kawakami-Nakajo \cite{KN}).
The precise definition of improper affine fronts
is given in \cite[Remark 4.3]{UY}.
The set of singular points $\Sigma_f$ of 
an improper affine front $f\colon{}\D_1\to\R^3$
is represented as
$\Sigma_f=\{z\in\D_1\,|\,|\rho(z)|=1\}$,
where $\rho(z):=dG/dF$.
In \cite{UY}, {\it weak completeness\/} of improper affine fronts
is introduced.
Then we have
\begin{theorem}\label{thm:affine}
 There exists a weakly complete improper affine front
 $f\colon{}\D_1\to\R^3$ whose image is bounded such that
 $\D_1\setminus\Sigma_f$ is open dense in $\D_1$,
 where $\Sigma_f$ is the set of the singular points of $f$.
\end{theorem}
\begin{proof}
 Let $\L=(F,G,H)$ be a complete bounded Legendrian immersion into $\C^3$.
 Since $\L$ is Legendrian, \eqref{eq:contact-c} yields that
 $dH = -F\,dG$.
 Here, by completeness of $\L$, the induced metric
 \[
   ds^2_\L=
     |dF|^2+|dG|^2+|dH|^2 
   = |dF|^2+|dG|^2+|FdG|^2 
   = |dF|^2 + (|F|^2+1)|dG|^2
 \]
 is complete.
 Moreover, since the image of $\L$ is bounded, we have
 \[
     ds^2_\L\leq C(|dF|^2+|dG|^2)\qquad
     (\text{$C> 0$ is a constant}).
 \]
 Thus, the metric
 \begin{equation}\label{eq:d-tau}
     d\tau^2 := |dF|^2 + |dG|^2
 \end{equation}
 is complete.
Hence, we have an improper affine front $f$
 using Martinez' representation formula \cite{Martinez}
 with respect to  $(F,G)$;
 \begin{align}
    f&= \left(G+\overline F,
            \frac{1}{2}(|G|^2-|F|^2)+
    \Re\left(GF-2\int F\,dG\right)\right)\label{eq:ia-map}\\
     &= \left(G+\overline F,
            \frac{1}{2}(|G|^2-|F|^2)+
            \Re\left(GF+2H\right)\right)\colon{}\D_1\to\R^3.
  \nonumber
 \end{align}
 Since $d\tau^2$ in \eqref{eq:d-tau} is complete,
 $f$ is weakly complete,
 by definition of weak completeness given in \cite{UY}.
 Boundedness of $f$ follows from that of $\L$.

 We next prove that $\D_1\setminus\Sigma_f$
 is open dense:
 If we take a initial immersion so that
 $|dG/dF|$ is not constant, then
 the resulting bounded weakly complete improper affine front
 has the same property, since 
 the singular set is characterized by
 $|dG/dF|=1$.
\end{proof}

\begin{remark}\label{rem:cpt}
 As mentioned in the introduction,
 there exist no compact flat fronts in $H^3$ (resp.\ $S^3_1$).
 In fact, \cite[Proposition 3.6]{KUY} implies non-existence
 of compact flat front in $H^3$.
 On the other hand, suppose that
 there exists a
 flat front $f\colon{}\Sigma^2\to S^3_1$
 where $\Sigma^2$ is a compact $2$-manifold.
 Then the unit normal vector $\nu$ of $f$ 
 induces a flat front $\nu\colon{}\Sigma^2\to H^3$,
 which makes a contradiction.

 Next, we shall show the non-existence 
 of compact improper affine fronts as mentioned in the introduction.
 An improper affine front $f\colon{}\Sigma^2\to\R^3$ defined on a Riemann
 surface $\Sigma^2$ is represented as in \eqref{eq:ia-map},
 where $F$ and $G$ are holomorphic functions on $\Sigma^2$.
 If $\Sigma^2$ is compact, $G+\overline F$ is equal to
 a constant $c\in \C$
 because it is harmonic.
 Then we have that $G=-\overline F+c$, where
 the left-hand side is holomorphic whereas
 the right-hand side is anti-holomorphic. 
 Hence we can conclude that 
 $F$ and $G$ are both constants.
 Hence $d\tau^2$ as in \eqref{eq:d-tau}
 vanishes identically, which contradicts
 to the definition of improper affine fronts.
\end{remark}

\appendix
\section{An application of Runge's Theorem}\label{app:runge}
To prove the Key Lemma, we prepare the following assertion,
which is an analogue of \cite[Lemma 4.1]{MUY}.

\begin{lemma}\label{lem:runge}
 Let 
 $\varphi=(\varphi_1,\varphi_2)$ be a non-degenerate pair
 of holomorphic functions on 
 $\overline{\D}_1$ {\rm(}see Definition~\ref{def:l-data}{\rm)},
 and let $\varepsilon>0$ and $N$ be a positive number and a sufficiently
 large integer $N$, respectively.
 Then for each $j$ {\rm($j=1,\dots,2N$)},
 there exists a non-degenerate pair
 $\tilde\varphi=(\tilde\varphi_1,\tilde\varphi_2)$
 of holomorphic functions  satisfying 
the following three conditions:
\begingroup
 \renewcommand{\theenumi}{{\rm(\alph{enumi})}}
 \renewcommand{\labelenumi}{{\rm(\alph{enumi})}}
 \begin{enumerate}
  \item\label{runge:1}
        $|\tilde\varphi-\varphi|<\frac{\varepsilon}{2N^2}$ on 
	$\overline\D_1\setminus \varpi_j$.
  \item\label{runge:2}
       $\displaystyle{
       |\tilde\varphi|\geq 
        \begin{cases}
	   C\,N^{3.5}\qquad &(\text{on $\omega_j$}), \\
	   C\,N^{-0.5}\qquad &(\text{on $\varpi_j$}),
	\end{cases}}$
       \newline
       where $C$ is a constant depending only on  $\varphi$.
  \item\label{runge:3}
       There exists a unit vector $\vect{u}=(u_1,u_2)\in\R^2$
       {\rm($|\vect{u}|=1$)}
       such that
       \[
	   \vect{u}\cdot (\varphi-\tilde\varphi)=0,\qquad
           |u_1|>1-\frac{2}{N},
       \]
       where we set 
       $\vect{v}\cdot\vect{w}=v_1w_1+v_2w_2$ for
       $\vect{v}=(v_1,v_2)$ and $\vect{w}=(w_1,w_2)$.
 \end{enumerate}
\endgroup
\end{lemma}

As pointed out in Remark \ref{rmk:new_pr} below,
we can prove the lemma as a modification of
\cite[Lemma 4.1]{MUY} for null holomorphic curves
in $\C^3$.
However, we believe that
the theory of Legendrian curves should be
established independently from the theory of
null curves. 
So we give here a self-contained proof
of the lemma, which might be convenient
for the readers.
In fact, our proof is easier than
that of \cite[Lemma 4.1]{MUY}.

\begin{proof}
 In the proof, the symbols $c_k$ ($k=1,2,\dots$)
 denote  suitable positive constants,
 which depend only on the initial data $\varphi$.
Since $\varphi$ has no zeroes,
we can take $\nu$ and $m$ such that
 \begin{equation}\label{eq:phi-inf}
     0<\nu\leq |\varphi|\leq m \qquad (\text{on $\overline{\D}_1$}).
 \end{equation}
 Set 
 \[
    \hat\varphi = (\hat\varphi_1,\hat\varphi_2)
    = \bigl((\cos t) \varphi_1+(\sin t) \varphi_2,
            -(\sin t) \varphi_1+(\cos t) \varphi_2\bigr)
 \]
 for $t\in \R$.
To prove the lemma, we need to prove that 
 one can choose
 $t\in [0,\frac{\pi}{2})$ such that
 \begin{equation}\label{eq:t-cond}
    \sin t \leq 
           \sqrt{\frac{2}{N}}
             \qquad |\hat\varphi_k|\geq\frac{\nu}{2\sqrt{N}}
            \qquad(\text{on $\varpi_j$}, k=1,2).
 \end{equation}
 In fact, if $|\varphi_k|\geq \nu/(2\sqrt{N})$ $(k=1,2)$ holds
 on $\varpi_j$, \eqref{eq:t-cond} holds obviously for $t=0$.
By exchanging the roles of $\varphi_1$ and $\varphi_2$,
we may assume that there exists $x\in\varpi_j$
 such that $|\varphi_1(x)|<\nu/(2\sqrt{N})$ without loss
of generality.
Here, notice that the diameter (as a subset of $\C=\R^2$) of $\varpi_j$
 satisfies $\diam_{\R^2}(\varpi_j)\leq c_1/N$, where $c_1$ is a 
 positive constant.
 Since the derivative of $\varphi\colon{}\overline{\D}_1\to\C^2$ 
 is bounded,
 it holds that,
 $\diam_{\C^2}(\varphi(\varpi_j))\leq c_2/N$ 
 for $c_2>0$.
 Then  it holds that
 \[
    |\varphi_1(y)|\leq  |\varphi_1(x)|+|\varphi(y)-\varphi(x)|
              < \frac{\nu}{2\sqrt{N}}+\frac{c_2}{N}
 \]
 for any $y\in \varpi_j$.
 Here, noticing  $|\varphi|^2\geq \nu^2$,
 we have
 \[
    |\varphi_2(y)| \geq \sqrt{\nu^2-|\varphi_1(y)|^2}
              \geq \sqrt{\nu^2
                    -
                    \left(
                      \frac{\nu}{2\sqrt{N}}+\frac{c_2}{N}
                    \right)^2}
              \geq \nu\left(1 - \frac{1}{4N}\right)
 \]
 since $N$ is sufficiently large.
We choose $t$ as
 $\sin t = \sqrt{\frac{2}{N}}$.
 Then noticing
 \[
     1-\frac{2}{N}\leq \cos t \leq 1-\frac{1}{N},
 \]
 we have
 \begin{align}
    |\hat\varphi_1(y)|&=
          |(\cos t)\varphi_1(y)+(\sin t)\varphi_2(y)|
         \geq (\sin t)|\varphi_2(y)|-(\cos t) |\varphi_1(y)|
           \label{eq:phi-est}\\
         & \geq  \frac{\sqrt{2}\nu}{\sqrt{N}}\left(1-\frac{1}{4N}\right)
               -\left(1-\frac{1}{N}\right)
              \left(\frac{\nu}{2\sqrt{N}}+\frac{c_2}{N}\right)
            \geq\frac{\nu}{2\sqrt{N}},
  \nonumber\\
  |\hat\varphi_2(y)|&
         \geq (\cos t)|\varphi_2(y)|-(\sin t) |\varphi_1(y)|
  \nonumber
          \\
         &\geq \left(1-\frac{2}{N}\right)
              \left(1-\frac{1}{4N}\right)\nu-
              \sqrt{\frac{2}{N}}
              \left(\frac{\nu}{2\sqrt{N}}+\frac{c_2}{N}\right)
         \geq \frac{\nu}{2\sqrt{N}}.\nonumber
 \end{align}
 Hence we have \eqref{eq:t-cond}.

 We now fix a real number $t$ in \eqref{eq:t-cond} and will
prove (a), (b) and (c):
 Since $\omega_j$ and $\D_1\setminus\varpi_j$ are compact set
 such that $\C\setminus\bigl(\omega_j\cup (\D_1\setminus\varpi_j)\bigr)$
 is connect, 
 Runge's theorem implies that there exists a holomorphic function 
 $h$  (cf.\ \cite[(4)]{Nadi}) 
 such that $h\neq 0$ on $\overline{\D}_1$ and
 \begin{equation}\label{eq:runge-fct}
  \left\{
  \begin{alignedat}{2}
   |h-2N^4| &< \frac{1}{2N^2}\qquad &&(\text{on $\omega_j$})\\
   |h-1|    &<   \frac{\varepsilon}{2N^2 m}
                           &&(\text{on
   $\overline{\D}_1\setminus\varpi_j$}),
  \end{alignedat}
  \right.
 \end{equation}
 where $m$ is as in \eqref{eq:phi-inf}.
We set
 \[
     \check\varphi = (\check\varphi_1,\check\varphi_2)
         := (\hat\varphi_1,h\hat\varphi_2).
 \]
 Then 
 \[
    \tilde\varphi 
    := \bigl((\cos t) \check\varphi_1-(\sin t) \check\varphi_2,
      (\sin t) \check\varphi_1+(\cos t) \check\varphi_2\bigr)
 \]
 satisfies the desired properties:
 In fact, 
 \[
    |\tilde\varphi-\varphi|
    = |\check\varphi-\hat\varphi|
    = |h-1|\,|\hat\varphi_2|
    < \frac{\varepsilon}{2N^2m}|\varphi| \leq \frac{\varepsilon}{2N^2}
 \]
 which implies \ref{runge:1}. On the other hand,
by \eqref{eq:phi-est},
 \[
   |\tilde\varphi|=|\check\varphi|
                  \geq |\check\varphi_1| = |\hat\varphi_1|\geq 
                  \frac{\nu}{2\sqrt{N}}
 \]
 holds on $\varpi_j$, which proves the first inequality of (b).

It holds on $\omega_j$ that
 \[
   |\tilde\varphi|=|\check\varphi|
      \geq |h|\,|\hat\varphi_2|
      \geq \bigl|2N^4- |h-2N^4| \bigr|\, |\hat\varphi_2|
      \geq \left(2N^4-\frac{1}{2N^2}\right)\frac{\nu}{2\sqrt{N}}
      \geq \frac{\nu}{2}N^{3.5}.
 \]
Hence we have the second inequality of \ref{runge:2}.
Finally, we set $\vect{u}=(\cos t,\sin t)$.
Then \ref{runge:3} holds.
\end{proof}

\begin{remark}\label{rmk:new_pr}
 Lemma~\ref{lem:runge} can be proved directly from
 the corresponding assertion for null curves in $\C^3$
 given in \cite[Lemma 4.1]{MUY} as follows:
 Let $\varphi=(\varphi_1,\varphi_2)$ be as in Lemma~\ref{lem:runge}.
 Since $\varphi_1,\varphi_2$ have no common zeros,
 there exists a holomorphic function
 $\varphi_3$ defined on $\overline{\D}_1$ such that
 $(\varphi_1)^2+(\varphi_2)^2+(\varphi_3)^2=0$.
 We apply \cite[Lemma 4.1]{MUY} for
 $\Phi:=(\varphi_1,\varphi_2,\varphi_3)$ and get
 a new Weierstrass data 
 $\tilde \Phi =(\tilde\varphi_1,\tilde\varphi_2,\tilde\varphi_3)$.
 Then $\tilde \varphi:=(\tilde\varphi_1,\tilde\varphi_2)$
 satisfies \ref{runge:1} which follows immediately
 from (a) of \cite[Lemma 4.1]{MUY}.
 Next, by the proof of \cite[Corollary B]{MUY} 
 it holds that
 \[
    4(|\tilde \varphi_1|^2+|\tilde \varphi_2|^2)
      \ge |\tilde \varphi_1|^2+|\tilde \varphi_2|^2+|\tilde \varphi_3|^2,
 \]
 which implies that \ref{runge:2} of our lemma
 follows from (b) of \cite[Lemma 4.1]{MUY}.
 \ref{runge:3} of our lemma does not follows from 
 (c) of \cite[Lemma 4.1]{MUY} directly.
 However, we can choose $u=(u_1,u_2,u_3)$
 in the proof of \cite[Lemma 4.1]{MUY} 
 in such a way that $u\in \R^3$ and $u_3=0$
 without loss of generality.
 So this gives a alternative proof of
 the lemma. 
\end{remark}

\end{document}